\documentclass[
final
]{dmtcs-episciences}

\usepackage[utf8]{inputenc}
\usepackage{subfigure}
\usepackage[round]{natbib}
\usepackage{amsmath}
\usepackage{amssymb}
\usepackage{amsthm}
\newtheorem{theorem}{Theorem}[section]
\newtheorem{lemma}[theorem]{Lemma}

\newtheorem{proposition}[theorem]{Proposition}
\newtheorem{conjecture}[theorem]{Conjecture}
\newtheorem{problem}[theorem]{Problem}
\theoremstyle{definition}
\newtheorem{definition}[theorem]{Definition}

\author{Ruixia Wang\affiliationmark{1}}
\title{A sufficient condition for a balanced bipartite digraph to be hamiltonian}
\affiliation{
  School of Mathematical Sciences, Shanxi University, China}
\received{2015-6-19}
\revised{2016-8-16}
\accepted{2017-10-3}
\begin{document}
\bibliographystyle{plainnat}

\publicationdetails{19}{2017}{3}{11}{1239}
\maketitle

\begin{abstract}
We describe a new type of sufficient condition for a balanced bipartite digraph to be hamiltonian. Let $D$ be a balanced bipartite digraph and $x,y$ be distinct vertices in $D$.  $\{x, y\}$ dominates a vertex $z$ if $x\rightarrow z$ and $y\rightarrow z$; in this case, we call the pair $\{x, y\}$ dominating. In this paper, we prove that a strong balanced bipartite digraph $D$ on $2a$ vertices contains a hamiltonian cycle if, for every dominating pair of vertices $\{x, y\}$, either $d(x)\ge 2a-1$ and $d(y)\ge a+1$ or $d(x)\ge a+1$ and $d(y)\ge 2a-1$. The lower bound in the result is sharp.
\end{abstract}

\section{Introduction}

In this paper, we consider finite digraphs without loops and multiple arcs. For the convenience of
the reader, we provide all necessary terminology and notation in one section, Section 2.  The cycle problems for digraphs
are of the central problems in graph theory and its applications \cite {bang}.  A digraph $D$ is called hamiltonian if it contains a
hamiltonian cycle, i.e., a cycle that includes every vertex of $D$.  There are many degree or degree sum conditions for hamiltonicity\cite {adamus1, adamus2, bang3, bang2, bermond, manoussakis,meyniel}. The following result of Meyniel for existence of hamiltonian cycles in digraphs is basic and famous.

\begin{theorem}\cite{meyniel} Let $D$ be a strong digraph on $n$ vertices where $n\ge 3$. If $d(x)+d(y)\ge 2n-1$ for all pairs of non-adjacent vertices $x, y$ in $D$, then $D$ is hamiltonian.\end{theorem}

In \cite{bang2}, Bang-Jensen, Gutin and Li described a type of sufficient condition for a digraph to be hamiltonian. Conditions of this type combine local structure of the digraph with conditions on the degrees of non-adjacent vertices. Let $x, y$ be distinct vertices in $D$. If there is an arc from $x$ to $y$ then we say that $x$ dominates $y$ and write $x\rightarrow y$. $\{x, y\}$ is dominated by a vertex $z$ if $z\rightarrow x$ and $z\rightarrow y$. Likewise, $\{x,y\}$ dominates a vertex $z$ if $x\rightarrow z$ and $y\rightarrow z$; in this case, we call the pair $\{x,y\}$ dominating.

\begin{theorem}\cite{bang2}\label{digloc} Let $D$ be a strong digraph. Suppose that, for every dominated pair of non-adjacent vertices $\{x,y\}$, either $d(x)\ge n$ and $d(y)\ge n-1$ or $d(x)\ge n-1$ and $d(y)\ge n$. Then $D$ is hamiltonian.
\end{theorem}

In \cite{bang2}, Bang-Jensen, Gutin and Li raised the following conjecture.

\begin{conjecture}\cite{bang2}\label{conjecture} Let $D$ be a strong digraph. Suppose that $d(x)+d(y)\ge 2n-1$ for every pair of dominating non-adjacent and every pair of dominated non-adjacent vertices $\{x, y\}$. Then $D$ is hamiltonian.
\end{conjecture}

Bang-Jensen, Guo and Yeo \cite{bang3} proved that, if we replaced the degree condition $d(x)+d(y)\ge 2n-1$ with $d(x)+d(y)\ge \frac{5}{2}n-4$ in Conjecture \ref{conjecture}, then $D$ is hamiltonian. They also proved additional support for Conjecture \ref{conjecture} by showing that every digraph satisfying the condition of Conjecture \ref{conjecture} has a cycle factor. From now on, Conjecture \ref{conjecture} is still open and seems quite difficult.

In \cite{adamus2}, Adamus, Adamus and Yeo  gave a Meyniel-type sufficient condition for hamiltonicity of a balanced bipartite digraph.

\begin{theorem}\label{hcycle2}\cite{adamus2} Let $D$ be a strong balanced bipartite digraph on $2a$ vertices. If  $d(u)+d(v)\ge 3a$ for every pair of non-adjacent vertices $u, v$ in $D$, then $D$ is hamiltonian. \end{theorem}

The main purpose of this note is to give a sharp sufficient condition for hamiltonicity of balanced bipartite digraphs similar to Theorem \ref{digloc}.

\begin{definition}
  Consider a balanced bipartite digraph $D$ on $2a$ vertices with $a\ge 2$. For $k\ge 0$, we will say that $D$ satisfies condition $B_k$ when  $$d(x)\ge 2a-k, d(y)\ge a+k\  \mbox{or}\ d(y)\ge 2a-k, d(x)\ge a+k,$$
  for any dominating pair of vertices $\{x,y\}$ in $D$.
\end{definition}

\begin{theorem}\label{mainresult} Let $D$ be a strong balanced bipartite digraph on $2a$ vertices where $a\ge 2$. If $D$ satisfies  condition $B_1$, then $D$ is hamiltonian.\end{theorem}

In Section 3, we shall prove Theorem \ref{mainresult} as a corollary of Theorem \ref{sharp}.

\begin{theorem}\label{sharp}
  Let $D$ be a strong balanced bipartite digraph on $2a$ vertices where $a\ge 2$. Suppose that, for every dominating pair of vertices $\{x, y\}$ either $d(x)\ge 2a-2, d(y)\ge a+1\  \mbox{or}\ d(y)\ge 2a-2, d(x)\ge a+1$. Then $D$ is either hamiltonian or isomorphic to one of the digraphs $H_1$, $H_2$ and $H_3$ (see Figure 1 ).
\end{theorem}

 It is not difficult to see that none of the digraphs $H_1$, $H_2$ and $H_3$ contains a hamiltonian cycle. Here, in the digraph $H_2$, $x_3$ may dominate $y_1$ or not.

%TeXCAD Picture [h1.tex]. Options:
%\grade{\on}
%\emlines{\off}
%\epic{\off}
%\beziermacro{\on}
%\reduce{\on}
%\snapping{\off}
%\quality{8.00}
%\graddiff{0.01}
%\snapasp{1}
%\zoom{4.0000}
\unitlength 1mm % = 2.85pt
\linethickness{0.4pt}
\ifx\plotpoint\undefined\newsavebox{\plotpoint}\fi % GNUPLOT compatibility
\begin{picture}(167.25,60)(90,120)
\put(154.5,164.5){\circle*{2.5}}
\put(177.5,164.5){\circle*{2.5}}
\put(200.5,164.5){\circle*{2.5}}
\put(154.5,143.5){\circle*{2.5}}
\put(177.5,143.5){\circle*{2.5}}
\put(200.5,143.5){\circle*{2.5}}
\put(149.5,168.5){\makebox(0,0)[cc]{$x_1$}}
\put(173.5,168.5){\makebox(0,0)[cc]{$x_2$}}
\put(196.5,168.5){\makebox(0,0)[cc]{$x_3$}}
\put(149.5,140.5){\makebox(0,0)[cc]{$y_1$}}
\put(173.5,140.5){\makebox(0,0)[cc]{$y_2$}}
\put(196.5,140.5){\makebox(0,0)[cc]{$y_3$}}
\put(177.25,143.5){\vector(0,1){20.25}}
%\vector(177.75,143.75)(155,163.75)
\put(155,163.75){\vector(-1,1){.07}}\multiput(177.75,143.75)(-.03836425,.033726813){593}{\line(-1,0){.03836425}}
%\end
%\vector(200.25,144)(178.75,164.75)
\put(178.75,164.75){\vector(-1,1){.07}}\multiput(200.25,144)(-.03495935,.033739837){615}{\line(-1,0){.03495935}}
%\end
%\vector(200.5,144.25)(155.75,165)
\put(155.75,165){\vector(-2,1){.07}}\multiput(200.5,144.25)(-.072764228,.033739837){615}{\line(-1,0){.072764228}}
%\end
%\vector(154.5,143.5)(199,165)
\put(199,165){\vector(2,1){.07}}\multiput(154.5,143.5)(.069749216,.03369906){638}{\line(1,0){.069749216}}
%\end
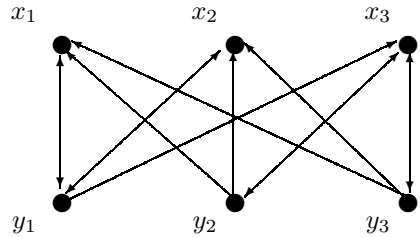
\put(177.5,130){\makebox(0,0)[cc]{The digraph $H_2$.}}
%\vector[both](178.75,144.25)(199.25,163.5)
\put(199.25,163.5){\vector(1,1){.07}}\put(178.75,144.25){\vector(-1,-1){.07}}\multiput(178.75,144.25)(.035901926,.033712785){571}{\line(1,0){.035901926}}
%\end
%\vector[both](200.75,145)(200.75,163.5)
\put(200.75,163.5){\vector(0,1){.07}}\put(200.75,145){\vector(0,-1){.07}}\put(200.75,145){\line(0,1){18.5}}
%\end
%\vector[both](154.25,145.5)(154.25,163)
\put(154.25,163){\vector(0,1){.07}}\put(154.25,145.5){\vector(0,-1){.07}}\put(154.25,145.5){\line(0,1){17.5}}
%\end
%\vector[both](155,144.75)(175.5,163.75)
\put(175.5,163.75){\vector(1,1){.07}}\put(155,144.75){\vector(-1,-1){.07}}\multiput(155,144.75)(.036347518,.033687943){564}{\line(1,0){.036347518}}
%\end
\put(122,164.5){\circle*{2.5}}
\put(99,164.5){\circle*{2.5}}
\put(122,143.5){\circle*{2.5}}
\put(99,143.5){\circle*{2.5}}
\put(94.25,167.25){\makebox(0,0)[cc]{$x_1$}}
\put(94.25,145.25){\makebox(0,0)[cc]{$x_2$}}
\put(126,167.25){\makebox(0,0)[cc]{$y_1$}}
\put(126,145.25){\makebox(0,0)[cc]{$y_2$}}
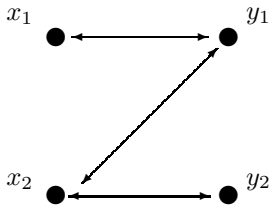
\put(106,130){\makebox(0,0)[cc]{The digraph $H_1$}}
%\vector[both](101.25,164.5)(119.25,164.5)
\put(119.25,164.5){\vector(1,0){.07}}\put(101.25,164.5){\vector(-1,0){.07}}\put(101.25,164.5){\line(1,0){18}}
%\end
%\vector[both](100.75,143.25)(119.5,143.25)
\put(119.5,143.25){\vector(1,0){.07}}\put(100.75,143.25){\vector(-1,0){.07}}\put(100.75,143.25){\line(1,0){18.75}}
%\end
%\vector[both](120.5,163)(102.5,145)
\put(102.5,145){\vector(-1,-1){.07}}\put(120.5,163){\vector(1,1){.07}}\multiput(120.5,163)(-.033707865,-.033707865){534}{\line(0,-1){.033707865}}
%\end
\end{picture}

%TeXCAD Picture [tu3.tex]. Options:
%\grade{\on}
%\emlines{\off}
%\epic{\off}
%\beziermacro{\on}
%\reduce{\on}
%\snapping{\off}
%\quality{8.00}
%\graddiff{0.01}
%\snapasp{1}
%\zoom{4.0000}
\unitlength 1mm % = 2.85pt
\linethickness{0.4pt}
\ifx\plotpoint\undefined\newsavebox{\plotpoint}\fi % GNUPLOT compatibility
\begin{picture}(167.25,35)(70,90)
\put(118,122){\circle*{2}}
\put(118,108){\circle*{2}}
\put(142,122){\circle*{2}}
\put(142,108){\circle*{2}}
\put(97,115){\circle*{2}}
\put(163,115){\circle*{2}}
%\vector(118,109.25)(140.25,121.5)
\put(140.25,121.5){\vector(2,1){.07}}\multiput(118,109.25)(.061126374,.033653846){364}{\line(1,0){.061126374}}
%\end
%\vector[both](98,116)(116,122)
\put(116,122){\vector(3,1){.07}}\put(98,116){\vector(-3,-1){.07}}\multiput(98,116)(.1011236,.03370787){178}{\line(1,0){.1011236}}
%\end
%\vector[both](97.75,114)(116,108.25)
\put(116,108.25){\vector(3,-1){.07}}\put(97.75,114){\vector(-3,1){.07}}\multiput(97.75,114)(.10672515,-.03362573){171}{\line(1,0){.10672515}}
%\end
%\vector[both](143.25,122.25)(161,115.25)
\put(161,115.25){\vector(3,-1){.07}}\put(143.25,122.25){\vector(-3,1){.07}}\multiput(143.25,122.25)(.08533654,-.03365385){208}{\line(1,0){.08533654}}
%\end
%\vector[both](143.5,108)(161.75,114)
\put(161.75,114){\vector(3,1){.07}}\put(143.5,108){\vector(-3,-1){.07}}\multiput(143.5,108)(.10252809,.03370787){178}{\line(1,0){.10252809}}
%\end
\put(145,125.25){\makebox(0,0)[cc]{$y_2$}}
\put(139.25,122.25){\vector(-1,0){19.5}}
%\vector(120.25,121.25)(139.25,108.75)
\put(139.25,108.75){\vector(3,-2){.07}}\multiput(120.25,121.25)(.051212938,-.033692722){371}{\line(1,0){.051212938}}
%\end
\put(139,107.75){\vector(-1,0){18}}
\put(114.75,104.75){\makebox(0,0)[cc]{$x_1$}}
\put(115.5,127){\makebox(0,0)[cc]{$x_2$}}
\put(143.25,104.5){\makebox(0,0)[cc]{$y_1$}}
\put(167.25,115.5){\makebox(0,0)[cc]{$x_3$}}
\put(92,114.75){\makebox(0,0)[cc]{$y_3$}}
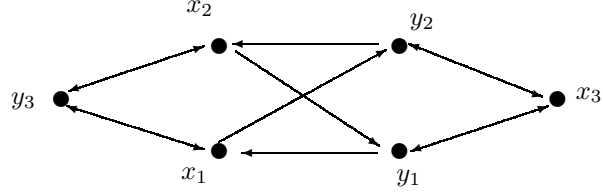
\put(128.75,95){\makebox(0,0)[cc]{Figure 1. The digraph $H_3$. }}
\end{picture}

It seems quite natural to ask whether a similar result to Conjecture \ref{conjecture} is true or not in balanced bipartite digraphs. In Section 4, we show that if a strong balanced bipartite digraph on $2a$ vertices such that $d(x)+d(y)\ge 3a$ for every pair of dominating and every pair of dominated vertices $\{x,y\}$, then $D$ must contain a cycle factor but may contain no hamiltonian cycle.

\section{Terminology and notation}

We shall assume that the reader is familiar with the standard terminology on
digraphs and refer the reader to \cite{bang} for terminology not defined here.
Let $D$ be a digraph with vertex set $V(D)$ and arc set $A(D)$.
For disjoint subsets $X$ and $Y$ of $V(D)$, $X\rightarrow Y$ means
that every vertex of $X$ dominates every
vertex of $Y$, $X\Rightarrow Y$ means that there is no arc from
$Y$ to $X$ and $X\mapsto Y$ means that both of $X\rightarrow Y$
and $X\Rightarrow Y$ hold.
%An arc $xy\in A(D)$ is called {\it
%asymmetrical} (resp. {\it symmetrical}) if $yx\notin A(D)$ (resp.
%$yx\in A(D)$).

For a vertex set $S\subset V(D)$, we denote by $N^+(S)$ the set of vertices in $V(D)$ dominated by the vertices of $S$; i.e. $N^+(S)=\{u\in V(D): vu\in A(D)\  \mbox{for some}\ v\in S\}.$
Similarly, $N^-(S)$ denotes the set of vertices of $V(D)$ dominating vertices of $S$; i.e.
 $N^-(S)=\{u\in V(D): uv\in A(D)\  \mbox{for some}\  v\in S\}.$
If $S=\{v\}$ is a single vertex, the cardinality of $N^+(v)$ (resp. $N^-(v)$), denoted by $d^+(v)$ (resp. $d^-(v)$) is called the {\it out-degree} (resp. {\it in-degree}) of $v$ in $D$. The degree of $v$ is $d(v)=d^+(v)+d^-(v)$.

Let $P=y_0y_1\ldots y_k$ be a path or a cycle of $D$. For $i\neq j$, $y_i, y_j\in V(P)$ we denote by $P[y_i, y_j]$ the {\it subpath} of $P$ from $y_i$ to $y_j$. If $0 < i\le k$ then the {\it predecessor } of $x_i$ on $P$ is the
vertex $x_{i-1}$ and is also denoted by $x^-_i$. If $0\le i<k$, then the {\it  successor} of $x_i$
on $P$ is the vertex $x_{i+1}$ and is also denoted by $x^+_i$. A digraph $D$ is said to be strongly connected or just strong, if for every pair $x, y$ of vertices of $D$, there is an $(x, y)$-path.

Let $C$ be a cycle in $D$. An $(x,y)$-path $P$ is a {\it $C$-bypass } if
$|V(P)|\ge 3$, $x\neq y$ and $V(P)\cap V(C)= \{x,y\}.$ The length of the path $C[x, y]$
is the gap of $P$ with respect to $C$. A {\it cycle factor} in $D$ is a collection of vertex-disjoint cycles $C_1, C_2, \ldots, C_t$ such that $V(C_1)\cup V(C_2)\cup \ldots \cup V(C_t)=V(D)$.

A digraph $D$ is {\it  bipartite} when $V(D)$ is a disjoint union of independent sets $V_1$ and $V_2$.  It is called {\it balanced} if $|V_1|=|V_2|$. A {\it matching} from $V_1$ to $V_2$ is an independent set of arcs with origin in $V_1$ and terminus in $V_2$ ($u_1u_2$ and $v_1v_2$ are independent arcs when $u_1\neq v_1$ and $u_2\neq v_2$). If $D$ is balanced, one says that such a matching is perfect if it consists of precisely $|V_1|$ arcs. A digraph $D$ is {\it semicomplete bipartite}, if the vertices of $D$
can be partitioned into two partite sets such that every partite set is
an independent set and for every pair $x, y$ of vertices from distinct
partite sets, $xy$ or $yx$ (or both) is in $D$.

\section{The main result}

\begin{theorem}\label{bifactor}\cite{gutin,haggkvis} A semicomplete bipartite digraph $D$ is hamiltonian if and only if it is strong and contains a cycle factor.
\end{theorem}

\begin{lemma}\label{cyclefactor} Let $D$ be a strong balanced bipartite digraph with partite sets $V_1$ and $V_2$ of cardinalities $a$ where $a\ge 2$.  Suppose that, for every dominating pair of vertices $\{x, y\}$, either $d(x)\ge 2a-2$ and $d(y)\ge a+1$ or $d(x)\ge a+1$ and $d(y)\ge 2a-2$  and suppose that $D$ is not isomorphic to the digraph $ H_2$.  Then $D$ contains a perfect matching from $V_1$ to $V_2$ and a perfect matching from $V_2$ to $V_1$. Moreover, $D$ contains a cycle factor.\end{lemma}

\begin{proof}  In order to prove that $D$ contains a perfect matching from $V_1$ to $V_2$ and a prefect matching from $V_2$ to $V_1$, by the K\"{o}nig-Hall theorem, it suffices to show that $|N^+(S)|\ge |S|$ for every $S\subset V_1$ and $|N^+(T)|\ge |T|$ for every $T\subset V_2$.

For a proof by contradiction, suppose that a non-empty set $S\subset V_1$ is such that $|N^+(S)|<|S|$. Then $V_2\setminus N^+(S)\neq\emptyset$. If $|S|=1$, write $S=\{x\}$, then $|N^+(S)|<|S|$ implies that $d^+(x)=0$. It is impossible in a strong digraph.  If $|S|=a$, then, for any $w\in V_2\setminus N^+(S)$,  the vertex $w$ is not dominated by any vertex of $D$ contradicting the fact that $D$ is strong. Thus, $2\le |S|\le a-1$, which implies $a\ge 3$ and $2a-2\ge a+1$ as well. Then $|N^+(S)|<|S|$ implies that there exist $x_1, x_2\in S$ and $y\in N^+(S)$ such that $\{x_1,x_2\} \rightarrow y$. Thus, $\{x_1, x_2\}$ is a dominating pair of vertices. By the hypothesis of this lemma, we assume, without loss of generality, that $d(x_1)\ge 2a-2$ and $d(x_2)\ge a+1$.  By
$$|N^+(S)|<|S|\le a-1,\eqno(3.1)$$
one gets $|N^+(S)|\le a-2$. Now we show that $|N^+(S)|=a-2$ and further $a=3$. To prove $|N^+(S)|=a-2$, it suffices to show that $|N^+(S)|\ge a-2$. Indeed, since there is no arc from $S$ to $V_2\setminus N^+(S)$, for any $w\in V_2\setminus N^+(S)$ and $x\in S$, $d(w)\le 2a-|S|\ \mbox{and} \ d(x)\le 2a-(a-|N^+(S)|).$
Thus,
$$2a-2\le d(x_1)\le 2a-(a-|N^+(S)|)=a+|N^+(S)|,\eqno(3.2)$$ that is, $|N^+(S)|\ge a-2$ and so $|N^+(S)|=a-2$, write $V_2\setminus N^+(S)=\{w_1, w_2\}$. It also follows that there must be equalities in all the estimates that led to (3.2).  In other words, $d(x_1)=2a-2$ and furthermore $x_1\rightarrow N^+(S)\rightarrow x_1$ and $\{w_1, w_2\}\mapsto x_1$. This, in turn, implies that $\{w_1, w_2\}$ is a dominating pair of vertices. By the hypothesis of this lemma, we assume, without loss of generality, that
$$d(w_1)\ge 2a-2\ \mbox{and}\  d(w_2)\ge a+1.\eqno(3.3)$$
 By (3.1) and $|N^+(S)|=a-2$, we can obtain that $|S|=a-1$. So, for any $w\in V_2\setminus N^+(S)$ and $x\in S$,
 $$d(w)\le a+1\ \mbox{and} \ d(x)\le 2a-2.\eqno(3.4)$$
 According to (3.3) and (3.4), we have that  $2a-2\le d(w_1)\le a+1$. So $a=3$. To convenience, write $V_1\setminus S=\{x_3\}$. By $d(x_2)\ge a+1=4$, (3.3) and (3.4), we get that $d(x_2)=d(w_1)=d(w_2)=4$. So  $y\rightarrow \{x_1, x_2\}\rightarrow y$, $x_3\rightarrow \{w_1, w_2\}\rightarrow x_3$ and $\{w_1, w_2\}\mapsto \{x_1, x_2\}$.  Since $D$ is strong, $y\rightarrow x_3$. Note that $D$ is isomorphic to the digraph $H_2$, contrary to our assumption.

 This completes the proof of existence of a perfect matching from $V_1$ to $V_2$. The proof for a perfect matching in the opposite direction is analogous. Observe that $D$ contains a cycle factor if and only if there exist both a perfect matching from $V_1$ to $V_2$ and a perfect matching from $V_2$ to $V_1$. Hence $D$ contains a cycle factor.
\end{proof}

\begin{lemma}\label{notmerge}
  Let $D$ be a strong bipartite digraph with partite sets $V_1$ and $V_2$. Let $C=x_0y_0x_1y_1\ldots x_{m-1}y_{m-1}x_0$ be a longest cycle in $D$, and let $uv$ be an arc in $D\setminus V(C)$, where $u$ and $x_i$ belong to the same partite set for $i\in \{0, 1,\ldots, m-1\}$. Then for every pair of vertices $y_i, x_{i+1}$ from $V(C)$ at most one of the arcs $y_iu$ and $vx_{i+1}$ belongs to $A(D)$. Furthermore, $d^-_{V(C)}(u)+d^+_{V(C)}(v)\le m$.
\end{lemma}

\begin{proof}
  If there exist $y_i, x_{i+1}\in V(C)$ such that $y_i\rightarrow u$ and $v\rightarrow x_{i+1}$, then $C$ and the arc $uv$ can be merged into a longer cycle than $C$ by deleting the arc $y_ix_{i+1}$ and adding the arcs $y_iu$ and $vx_{i+1}$. This would contradict the fact that $C$ is a longest cycle in $D$, so at most one of the arcs $y_iu$ and $vx_{i+1}$ belongs to $A(D)$. There is precisely $m$ of such pairs. By accounting for the arcs $y_iu$ and $vx_{i+1}$, we get the required estimate $d^-_{V(C)}(u)+d^+_{V(C)}(v)\le m$.
\end{proof}

\begin{lemma}\label{4cycle} Let $D$ be a strong balanced bipartite digraph with partite sets $V_1$ and $V_2$ of cardinalities $a$ where $a\ge 3$. Suppose that, for every dominating pair of vertices $\{x, y\}$, either $d(x)\ge 2a-2$ and $d(y)\ge a+1$ or $d(x)\ge a+1$ and $d(y)\ge 2a-2$. Then $D$ contains a cycle of length at least $4$.\end{lemma}

\begin{proof}  If $D$ is isomorphic to the digraph $H_2$, then obviously $D$ contains a cycle of length 4. If not, then suppose, on the contrary, that $D$ contains no cycle of length more than or equal to 4. So $D$ is non-hamiltonian as $a\ge 3$. By Lemma \ref{cyclefactor}, $D$ contains a cycle factor. Let $C_1, C_2, \ldots, C_t$ be a minimal cycle factor. Then the length of every $C_i$ is 2 and $t=a$. Write $C_i=x_iy_ix_i$, where $x_i\in V_1$ and $y_i\in V_2$ for $i=1,2,\ldots,a$.  By Lemma \ref{notmerge}, $d_{V(C_j)}(x_i)+d_{V(C_j)}(y_i)\le 2$, for every $i\in \{1,2,\ldots, a\}$ and $j\in \{1, 2,\ldots, a\}\setminus \{i\}$. Thus $d(x_i)+d(y_i)\le 2(a-1)+4=2a+2$.

If there exists a vertex $y_j\in V_2\setminus \{y_i\}$ such that between $x_i$ and $y_j$ form a 2-cycle. Without loss of generality, assume that  $i=1$ and $j=2$. Note that  $\{x_1,x_2\}$ and $\{y_1,y_2\}$ are both dominating pairs of vertices. Thus, by assumption
  \begin{align*}
2(3a-1) &  \le d(x_1)+d(x_2)+d(y_1)+d(y_2)\\
        &  =   d(x_1)+d(y_1)+d(x_2)+d(y_2)\\
        &  \le 2(2a+2),\end{align*}
which implies that $a\le 3$. By the hypothesis of this lemma, $a\ge 3$ and so $a=3$.  Further, there must be equalities in all the estimates, i.e. $d(x_i)=d(y_i)=4$, for $i=1,2$. By Lemma \ref{notmerge}, $x_2$ and $y_1$ are not adjacent, so $y_1\rightarrow x_3\rightarrow y_1$ and $y_3\rightarrow x_2\rightarrow y_3$. However, $x_1y_1x_3y_3x_2y_2x_1$ is a hamiltonian cycle, a contradiction.

Now assume that there exist no such vertices.  Combining this with Lemma \ref{notmerge}, for any two distinct cycles $C_i$ and $C_j$, we have that $C_i\Rightarrow C_j$ or $C_j\Rightarrow C_i$. Since $D$ is strong, $t\ge 3$. Without loss of generality, assume that $C_3\Rightarrow C_1$ and $C_1\Rightarrow C_2$ and further assume that $x_1\rightarrow y_2$.  Note $\{x_1, x_2\}$ is a dominating pair of vertices. By the hypothesis of this lemma, $d(x_1)\ge 2a-2$, $d(x_2)\ge a+1$ or $d(x_2)\ge 2a-2$, $d(x_1)\ge a+1$. Observe that for every $i\in \{1,2,\ldots, t\}$, $d(x_i)\le a+1$ and $d(y_i)\le a+1$. Thus, $2a-2\le a+1$, i.e. $a\le 3$ and so $a=3$. From this, it is not difficult to get that for any dominating pair vertices $\{u, v\}$, $d(u)=d(v)=4$. So, $d(x_1)=d(x_2)=4$, which implies that $y_3\rightarrow x_1$. This, in turn, implies that $\{y_1, y_3\}$ is a dominating pair. Thus, $d(y_1)=d(y_3)=4$ and further $y_1\rightarrow x_2$ and $x_3\rightarrow y_1$.  Since $D$ is strong, it must be $x_2\rightarrow y_3$. However $y_3x_3y_1x_1y_2x_2y_3$ is a hamiltonian cycle. a contradiction.
\end{proof}

{\bf\noindent The proof of Theorem \ref{sharp}}.

\begin{proof} Let $V_1$ and $V_2$ denote the partite sets of $D$. By assumption of this theorem, a dominating pair $\{x, y\}$ means $d(x)\ge 2a-2$, $d(y)\ge a+1$, or  $d(y)\ge 2a-2$, $d(x)\ge a+1$. We will implicity use this in the remainder of the paper.
Suppose $a=2$. Denote $V_1=\{x_1, x_2\}$ and $V_2=\{y_1, y_2\}$. Since $D$ is strong, the in-degree of every vertex is at least one. If the in-degree of every vertex is one in $D$, then clearly $D$ contains a hamiltonian cycle. Now assume that there exists $z\in V(D)$ such that $d^-(z)\ge 2$, it is to say that there exist dominating pair of vertices in $D$. Without loss of generality, assume that $\{x_1, x_2\}$ is a dominating pair and furthermore $d(x_1)\ge 2a-2=2$ and $d(x_2)\ge a+1=3$. This  means that $x_2$ and every vertex of $V_2$ are adjacent. If $x_1$ and every vertex of $V_2$ are adjacent, then $D$ is  a semicomplete bipartite digraph. Using Theorem \ref{bifactor} and Lemma \ref{cyclefactor}, $D$ is hamiltonian. If $x_1$ and one of $y_1$ and $y_2$, say $y_2$, are not adjacent, then by $d(x_1)\ge 2$, we have that $y_1x_1y_1$ is a 2-cycle. Since $D$ is strong, we can deduce that  $x_2\rightarrow \{y_1, y_2\}\rightarrow x_2$. Note that $D$ is isomorphic to the digraph $H_1$.

Now suppose that $a\ge 3$. In this case, $2a-2\ge a+1$ and equality holds only if $a=3$. Suppose that $D$ is non-hamiltonian. Let $C=x_0y_0x_1y_1\ldots$ $x_{m-1}y_{m-1}x_0$ be a longest cycle in $D$, where $x_i\in V_1$ and $y_i\in V_2$ for $i\in\{0,1,\ldots, m-1\}$. Lemma \ref{4cycle} implies that $m\ge 2$. From now on, all subscripts appearing in this proof are taken modulo $m$.

We first show that $D$ contains a $C$-bypass. Assume $D$ does not have one. Since $D$ is strong, it must contain a cycle $Z$ such that $|V(Z)\cap V(C)|=1$.  Without loss of generality, assume that $V(Z)\cap V(C)=\{x_0\}$. Let $z$ be the predecessor of $x_0$ on $Z$. Since $\{z, y_{m-1}\}$ is a dominating pair of vertices, we have $d(z)\ge 2a-2$ and $d(y_{m-1})\ge a+1$ or $d(z)\ge a+1$ and $d(y_{m-1})\ge 2a-2$. Since $D$ contains no $C$-bypass, we have $d_{V(C)\setminus \{x_0\}}(z)=0$ and $d_{V(Z)\setminus \{x_0\}}(y_{m-1})=0$.  Note that $|(\{z, y_{m-1}\}, x)\cup (x, \{z, y_{m-1}\})|\le 2$ for every $x\in V_1\setminus (V(C)\cup V(Z))$. Denote
$|V_1\cap (V(Z)\setminus \{x_0\})|=p$. Hence $d(z)+d(y_{m-1})\le 2(a-p-m)+2p+2+2m=2a+2$.
This follows $3a-1\le d(y_{m-1})+d(z)\le 2a+2$.  Using this inequalities with $a\ge 3$, we obtain that $a=3$. In addition, the above inequalities become equalities, which implies that $d(z)=4$. Clearly, $m=2$, write $V_1\cap (V(D)\setminus V(C))=\{x\}$. By $d_{V(C)\setminus \{x_0\}}(z)=0$ and $d(z)=4$, we have $zxz$ is a 2-cycle. Since $D$ contains no $C$-bypass, we have $d_{V(C)}(x)=0$, that is $d(x)=2$. However, $\{x, x_0\}\rightarrow z$ implies that $\{x, x_0\}$ is a dominating pair of vertices. By assumption, $d(x)\ge a+1=4$, a contradiction. Therefore, $D$ contains a $C$-bypass.

Let $P=u_1u_2\ldots u_s$ be a $C$-bypass $(s\ge 3)$. Suppose also that the gap of $P$ is minimum among the gaps of all $C$-bypass. Denote $C'=C[u_1^+, u_s^-]$, where $u^+_1$ is the successor of $u_1$ on $C$ and $u^-_s$ is the predecessor of $u_s$ on $C$. Since $C$ is a longest cycle of $D$, we have that $|V(C')|\ge s-2$. Because $D$ is a bipartite digraph, when $s$ is odd, $u_1$ and $u_s$ belong to the same partite set; when $s$ is even, $u_1$ and $u_s$ belong to distinct partite sets.  Denote $R=V(D)\setminus (V(C)\cup V(P))$.  Noting that $\{u_{s-1}, u_s^-\}$ is a dominating pair of vertices, by assumption, $$d(u_{s-1})\ge 2a-2, d(u_s^-)\ge a+1 \ \mbox{or}\ d(u_{s-1})\ge a+1, d(u_s^-)\ge 2a-2.\eqno(3.5)$$

The following two claims will be very useful in the remaining proof.

\vskip 0.3cm
{\noindent\bf Claim 1.} For any $x\in R$, $|(\{u_{s-1}, u_s^-\}, x)\cup (x, \{u_{s-1}, u_s^-\})|\le 2$.

\begin{proof} Clearly, $u_{s-1}$ and $u^-_s$ belong to the same partite set.  It suffices to prove the case when $x$ and $u_{s-1}$ belong to distinct partite sets.  If $u_{s-1}\rightarrow x$, then $x\nrightarrow u^-_s$, for otherwise
$P[u_1, u_{s-1}]xu^-_s$ is also a $C$-bypass, which gap is strictly less than $P$, a contradiction.  If $x\rightarrow u_{s-1}$, then $u^-_s\nrightarrow x$, for otherwise $u^-_sxu_{s-1}C[u_s, u^-_s]$ is a longer cycle than $C$, a contradiction. Hence the claim holds.   \end{proof}

{\noindent\bf Claim 2.} Every vertex on $C'$ is not adjacent to any vertex on $P[u_2, u_{s-1}]$.

\begin{proof} Since  $P$ has the minimum gap, the claim is obvious. \end{proof}

Now we divide the proof into two cases to consider.

\vskip0.2cm
 {\noindent\bf Case 1. }  $|V(C')|\ge 2$.
 \vskip0.2cm

To complete the proof, we will first  give the following useful observation.
 \vskip0.2cm
{\noindent\bf Observation 1.}  For any $x\in V(D)\setminus V(C)$, $d_{V(C)}(x)\le m$.

\begin{proof}
 Assume, without loss of generality, that $x\in V_1$. Let $d^-_{V(C)}(x)=t$ and denote $N^-_{V(C)}(x)=\{y_{i_1}, y_{i_2}, \ldots, y_{i_t}\}$. According to $|V(C')|\ge 2$ and the fact that $P$ has the minimum gap, $x\nrightarrow y_{i_j+1}$ for every $j\in \{1,2,\ldots, t\}$. Hence, $d_{V(C)}(x)=d^-_{V(C)}(x)+d^+_{V(C)}(x)\le t+(m-t)=m$.
\end{proof}

We may assume, without loss of generality, that $d(u^-_s)\ge 2a-2$ and $d(u_{s-1})\ge a+1$.  In fact, let $x\in V(D)\setminus V(C)$ be any. By Observation 1, $d_{V(C)}(x)\le m$. So $d(x)\le m+2(a-m)=2a-m$. If $m\ge 3$, then $d(x)\le 2a-3$, in particular, $d(u_{s-1})\le 2a-3$. Combining this with (3.5), we have that $d(u^-_s)\ge 2a-2$ and $d(u_{s-1})\ge a+1$. If $m=2$, then $|V(C')|\ge 2$ implies that $|V(C')|=s-2=2$. By symmetry, we may assume, without loss of generality, that $d(u^-_s)\ge 2a-2$ and $d(u_{s-1})\ge a+1$.

If $s\ge 6$, then, by Claim 2,  $d(u_s^-)\le 2a-4$, a contradiction to $d(u^-_s)\ge 2a-2$. So, we assume from now on that $s\le 5$.  If $s\ge 4$, then,  by Claim 2, we have that $d(u^-_s)\le 2a-2$ and so $d(u^-_s)=2a-2$.

Suppose $s=5$.   In this case, $u_1$ and $u_s$ belong to the same partite set. Now,  without loss of generality, assume that $u_1=x_0$ and  $u_s=x_r$. Clearly, $R\cap V_1\neq\emptyset$.  By $|V(C')|\ge s-2$, we have that $r\ge 2$. By the above argument, $d(y_{r-1})=2a-2$ and $y_{r-1}$ and every vertex of $V_1\setminus \{u_3\}$ form a 2-cycle,  in particular, for any $x\in R\cap V_1$,  $x$ and $y_{r-1}$ form a 2-cycle. By Claim 1, $u_4$ and $x$ are not adjacent. Combining this with Observation 1, we have $d(x)\le m+2(a-m-1)=2a-m-2\le 2a-4$.  Note that  $\{x_{r-1}, x\}\rightarrow y_{r-1}$, i.e. $\{x, x_{r-1}\}$ is a dominating pair. Using Claim 2,  we can obtain that $d(x_{r-1})\le 2a-4$,  contradicting the fact that $\{x, x_{r-1}\}$ is a dominating pair.

Suppose  $s=4$.  In this case,  $|V(C')|$ is even and $u_1$ and $u_s$ belong to distinct partite sets. Now assume, without loss of generality, that $u_1=x_0$ and $u_s=y_r$.  Obviously, $r\ge 1$.   By the above argument,  $d(x_r)=2a-2$ and $x_r$ and every vertex of $V_2\setminus \{u_2\}$ form a  2-cycle. By Claim 1, $u_3$ and every vertex of $R\cap V_2$ are not adjacent.
 This together with Observation 1 implies that  $d(u_3)\le m+2$.  Combining this with $d(u_3)\ge a+1$, we have $a=m+1$. Thus, $R=\emptyset$ and $d_{V(C)}(u_3)=m$.

 Assume that $r=1$.  Note that  $\{y_0, y_1\}$ is a dominating pair. Since $2a-2\ge a+1$, by assumption, $d(y_0)\ge a+1$. For any $i\in \{2, \ldots, m-1\}$, if $x_i\rightarrow y_0$, then $x_iy_0x_1C[y_{i+1}, x_0]u_2u_3C[y_1, x_i]$ is a hamiltonian cycle, a contradiction; if $y_0\rightarrow x_i$, then $y_{i-1}x_1y_0C[x_i, x_0]u_2u_3C[y_1, y_{i-1}]$ is a hamiltonian cycle, a contradiction. Thus $y_0$ and $x_i$ are not adjacent. Furthermore, $y_0\nrightarrow x_0$, for else $u_2u_3C[y_1, y_{m-1}]x_1y_0x_0u_2$ is a hamiltonian cycle, a contradiction. From this we have that $d(y_0)\le 3$. However, $a+1\le d(y_0)\le 3$ implies that $a\le 2$, a contradiction.

 Assume that $r\ge 2$, i.e. $|V(C')|\ge 4$.  Denote $d^-_{V(C)}(u_3)=t$.  If $t=0$, then, by Claim 2, $d_{V(C)}(u_3)\le m-2$, a contradiction to $d_{V(C)}(u_3)=m$. Next assume $t\ge 1$.
  Since $P$ has the minimum gap with respect to $C$, if $y_i\rightarrow u_3$, then $u_3\nrightarrow y_{i+1}$ and $u_3\nrightarrow y_{i+2}$. Hence $d_{V(C)}(u_3)=d_{V(C)}^-(u_3)+d_{V(C)}^+(u_3)\le t+(m-2t)=m-t\le m-1$, a contradiction.

Suppose  $s=3$.  In this case, $u_1$ and $u_s$ belong to the same partite set. Now assume that $u_1=x_0$ and $u_s=x_r$. Clearly, $R\cap V_1\neq \emptyset$. Since $|V(C')|\ge 2$,  we have $r\ge 2$, which means that $m\ge 3$. Let $x\in V(D)\setminus V(C)$ be any. By Observation 1, $d_{V(C)}(x)\le m$. So $d(x)\le m+2(a-m)=2a-m\le 2a-3$. This means that any pair of vertices in $V(D)\setminus V(C)$ cannot form a dominating pair. Hence $d^-_{V(D)\setminus V(C)}(x)\le 1$ and so $d(x)\le m+a-m+1=a+1$, which, in turn, implies that if $x$ and some vertex of $V(D)$ form a dominating pair, then $d(x)=a+1$ and further $d^-_{V(D)\setminus V(C)}(x)=1$ and $d^+_{V(D)\setminus V(C)}(x)=a-m$. So $d(u_2)=a+1$. Furthermore, $d^-_R(u_2)=1$ and $d^+_R(u_2)=a-m$. So there exists $x'\in R\cap V_1$ such that $u_2$ and $x'$ form a 2-cycle. Then $\{x',x_0\}\rightarrow u_2$ implies that $\{x',x_0\}$ is a dominating pair. So $d(x')=a+1$.
Now we show that $d_{V(C)}(u_2)+d_{V(C)}(x')\le 2(m-1)$. In fact, since $C$ is a longest cycle and $P$ has the minimum gap with respect to $C$, we can obtain that, for any $x_i\in N^-(u_2)\cap V(C)$, $x'\nrightarrow y_i$ and $x'\nrightarrow y_{i+1}$. Denote $d^-_{V(C)}(u_2)=t$. Clearly, $t\ge 1$. Then, $d_{V(C)}^-(u_2)+d_{V(C)}^+(x')\le t+(m-2t)=m-t\le m-1$. Similarly, for any $x_i\in N^+(u_2)\cap V(C)$, $y_{i-2}\nrightarrow x'$ and $ y_{i-1}\nrightarrow x'$. Hence $d_{V(C)}^+(u_2)+d_{V(C)}^-(x')\le m-1$. Add these two inequalities, we obtain $d_{V(C)}(u_2)+d_{V(C)}(x')\le 2(m-1)$, which implies the desired inequality. However,  $2(a+1)=d(u_2)+d(x')\le 2(m-1)+2(a-m+1)=2a$ is  a contradiction.

\vskip0.2cm
 {\noindent\bf Case 2. }  $|V(C')|=1$.
 \vskip0.2cm

 In this case, $s=3$  and $u_1$ and $u_s$ belong to the same partite set. Assume that  $u_1=x_0$ and $u_s=x_1$. By symmetry and (3.5), without loss of generality, assume that $d(u_2)\ge 2a-2$ and $d(y_0)\ge a+1$. Clearly, $R\cap V_1\neq \emptyset$.

 \vskip0.2cm
 {\noindent\bf Subcase 2.1.}  There exists $u\in R\cap V_1$ such that $u_2$ and $u$ are not adjacent.
 \vskip0.2cm

 First $d(u_2)\ge 2a-2$ and the hypothesis of the subcase imply  $d(u_2)=2a-2$. Moreover,  $u_2$ and every vertex of $V_1\setminus \{u\}$ form a 2-cycle, which implies that every pair of vertices of $V_1\setminus \{u\}$ form a dominating pair. Thus, the degree of every vertex in  $V_1\setminus \{u\}$ is greater or equal to $a+1$.

  If $a=3$, then $2a-2=a+1=4$. Because $D$ is strong, there exists $y_i\in V(C)\cap V_2$ such that $u\rightarrow y_i$.  Then $\{x_i,u\}\rightarrow y_i$ means that  $\{x_i, u\}$ is a dominating pair and so $ d(u)\ge a+1=4$.   Since  $u_2$ and $u$ are not adjacent,  we obtain that  $\{y_0, y_1\}\rightarrow u\rightarrow \{y_0,y_1\}$.  If $y_0\rightarrow x_0$, then $y_0x_0u_2x_1y_1uy_0$ is a hamiltonian cycle, a contradiction. Hence $x_0\mapsto y_0$. Similarly, we can obtain that $y_0\mapsto x_1$, $x_1\mapsto y_1$ and $y_1\mapsto x_0$. Note that $D$ is isomorphic to the digraph $H_3$.

  Now assume $a\ge 4$.   Suppose $|R\cap V_1|=1$.  We first show that $d(u)\le m+1$. Indeed, if there exist $y_i, y_{i+1}\in V(C)$ such that $y_i\rightarrow u$ and $u\rightarrow y_{i+1}$, then $d(x_{i+1})\le  2a-3$. For $j\in \{0, 1, \ldots, m-1\}\setminus \{i+1\}$,  $x_{i+1}\nrightarrow y_j$ otherwise, $x_{i+1}C[y_j, y_i]uC[y_{i+1}, x_j]u_2x_{i+1}$ is a hamiltonian cycle, a contradiction.  In addition, $y_{i+1}\nrightarrow x_{i+1}$, for else $y_{i+1}x_{i+1}u_2C[x_{i+2}, y_i]uy_{i+1}$ is a hamiltonian cycle, a contradiction.  Hence $d(x_{i+1})\le 2a-m\le 2m-3$.  Therefore, there  exists at most one pair of $\{y_i, y_{i+1}\}$ such that $y_i\rightarrow u\rightarrow y_{i+1}$, which implies  $d_{V(C)}(u)\le m+1$.
Because $D$ is strong, there exists $y_i\in V(C)\cap V_2$ such that $u\rightarrow y_i$.  Then $\{x_i,u\}\rightarrow y_i$ means that  $\{x_i, u\}$ is a dominating pair and so $ d(u)\ge a+1$ as $2a-2\ge a+1$, contrary to  $d(u)\le m+1$.

Suppose  $|R\cap V_1|\ge 2$. First we claim that  $a=2m+1$ and the degree of every vertex of $(R\cap V_1)\setminus \{u\}$ is equal to $a+1$.  In fact,  let $x\in (R\cap V_1)\setminus \{u\}$ be any. Similar to  Claim 1, we can deduce that $y_i$ and $x$ are not adjacent, for every $y_i\in V(C)\cap V_2$.  Recalling that $d(x)\ge a+1$,  we have  $a+1\le d(x)\le 2(a-m)$, i.e. $a\ge 2m+1$. Recalling that $d(y_0)\ge a+1$, we have $a+1\le d(y_0)\le 2m+2$, i.e. $a\le 2m+1$.  From these,  we get that  $a=2m+1$ and equality hold everywhere. So $d(x)=a+1$.
Then $a=2m+1$ implies that $a$ is odd and  further $|R\cap V_1|\ge 3$. This is a contradiction to the fact that every pair of vertices of $(R\cap V_1)\setminus \{u\}$ form a dominating pair and  the degree of every vertex in  $(R\cap V_1)\setminus \{u\}$ is equal to $a+1$.

 \vskip0.2cm
 {\noindent\bf Subcase 2.2. }  $u_2$ and every vertex of  $R\cap V_1$ are adjacent.
 \vskip0.2cm
Next we divide the subcase into three subcases.

%{\bf\noindent Observation 1.}  If there exists $x_i, x_{i+1}\in V(C)\cap V_1$ such that $x_i\rightarrow u_2\rightarrow x_{i+1}$, then $N(u_2)\cap (R\cap V_1)$ and $y_i$ are not adjacent.
%
%\begin{proof}
%  Since $C$ is the longest cycle, the observation is obvious.
%\end{proof}
%
%{\bf\noindent Observation 2.} For any $y_i\in V(C)\cap V_2$, if $u\rightarrow y_i$, then $x_i\nrightarrow u_2$; if $y_i\rightarrow u$, then $u_2\nrightarrow x_{i+1}$.
%
%\begin{proof}
%  In fact, if $x_i\rightarrow u_2$, then $x_iu_2uC[y_i, x_i]$ is a longer cycle than $C$, a contradiction; if $u_2\rightarrow x_{i+1}$, then $y_iuu_2C[x_{i+1},y_i]$ is a longer cycle than $C$, a contradiction.
%\end{proof}
%\vskip 0.2cm
%{\bf\noindent Observation 2.}  Let $xyx$ be a 2-cycle in $D-V(C)$, where $x\in V_1$ and $y\in V_2$. For any $x_i, y_i\in V(C)$, at most one of $x_iy$ and $xy_i$ belongs to $A(D)$ and at most one of $y_{i-1}x$ and $yx_i$ belongs to $A(D)$. Furthermore,  $d_{V(C)}(x)+d_{V(C)}(y)\le 2m$.
%
%\begin{proof} Since $C$ is the longest cycle, the observation is obvious.
%\end{proof}

%{\bf\noindent Observation 4.   }  Let $Q$ be a longest cycle in $D$ and $D-Q=xyx$ be  a 2-cycle. If  $d(x)+d(y)\ge 3a-1$, then $a=3$.
%
%\begin{proof}
% By Observation 3, $d_{V(Q)}(x)+d_{V(Q)}(y)\le 2m$. Then $3a-1\le d(x)+d(y)\le 2m+4$, that is, $a\le 3$. Furthermore, $a=3$ as we have assumed $a\ge 3$.
%  \end{proof}

 \vskip0.2cm
 {\noindent\bf Subcase 2.2.1.  }  There exists a vertex $u\in R\cap V_1$ such that $u_2$ and $u$ form a 2-cycle.
 \vskip0.2cm

Then $\{u, x_0\}\rightarrow u_2$  implies that  $$d(x_0)\ge 2a-2, d(u)\ge a+1\ \mbox{or}\  d(u)\ge 2a-2, d(x_0)\ge a+1. \eqno(3.6)$$

Assume  $a=3$.   Consider the cycle $u_2x_1y_1x_0u_2$ and note that $y_0$ and $u$ are not adjacent.  Similar to Subcase 2.1, we can obtain that  $D$ is isomorphic to the digraph $H_3$. Now assume that $a\ge 4$.

 If $d(u)\ge 2a-2$, then, by Lemma \ref{notmerge},  $4a-4\le d(u_2)+d(u)\le 2m+4(a-m)=4a-2m\le 4a-4$. Hence, equalities  hold everywhere, in particular,  $m=2$ and $d_{R}(u_2)=2(a-m)$.  In other words,  $u_2$ and every vertex of $R\cap V_1$ form a 2-cycle. By Claim 1, $y_0$ and every vertex of $R\cap V_1$ are not adjacent.  Hence  $d(y_0)\le 4$. This together with $d(y_0)\ge a+1$ implies that $a\le 3$, a contradiction.
Hence, we assume that $d(u)<2a-2$. In fact, we may assume that for any $x\in R\cap V_1$, if $x$ and $u_2$ form a 2-cycle, then $d(x)< 2a-2$.
 This also implies that  $u$ is the unique vertex in $R\cap V_1$  which forms a 2-cycle with $u_2$.    Using Lemma \ref{notmerge},  $3a-1\le d(u_2)+d(u)\le 2m+4(a-m)=4a-2m,$ i.e.  $2m\le a+1$.

 If  $2m=a+1$, then equalities hold everywhere.  It follows that $d_{R}(u_2)=2(a-m)$.  In other words,  $u_2$ and every vertex of $R\cap V_1$ form a 2-cycle.  Hence  $|R\cap V_1|=1$. However,  $3a-1\le d(u)+d(u_2)\le 2m+4$ implies that $a\le 3$, a contradiction.

If  $2m<a+1$, then $d(y_0)\ge a+1$ implies that $y_0$ must be adjacent to some vertex of $R\cap V_1$.  Since $y_0$ and $u$ are not adjacent, we have $|R\cap V_1|\ge 2$. Since $d(u_2)\ge 2a-2$, we have that $|R\cap V_1|\le 3$.

Suppose $|R\cap V_1|=3$, i.e. $a=m+3\ge 5$.  Write $R\cap V_1=\{u, w_1, w_2\}$ and $R\cap V_2=\{u_2, v_1,v_2\}$. By the above argument, we can conclude that $u_2$ and every vertex of $V(C)\cap V_1$ form a 2-cycle.   Similar  to Claim 1, we can obtain that $u$ are not adjacent to every $y_i\in V(C)$. Thus, $d(u)\le 6$. This together with $d(u)\ge a+1 $ implies that $a\le 5$ and so $a=5$ and $d(u)=a+1=6$. Hence,  $u$ and every vertex of $R\cap V_2$ form a 2-cycle.   In addition, according to  $d(y_0)\ge a+1=6$ and Claim 1,  either$\{u_2, y_0\}\mapsto w_1$ or $w_1\mapsto \{u_2, y_0\}$ and  $y_0\rightarrow \{x_0, x_1\}\rightarrow y_0$. Without loss of generality, assume that $\{u_2, y_0\}\mapsto w_1$. Since $D$ is strong, there exists a vertex in $V_2\setminus \{u_2, y_0\}$ dominated by $w_1$.  If $w_1\rightarrow y_1$, then $w_1y_1x_0u_2x_1y_0w_1$ is a longer cycle than $C$, a contradiction. If $w_1$ dominates one of $\{v_1, v_2\}$, say $v_1$, then $w_1v_1uu_2x_1y_0w_1$ is a longer cycle than $C$, a contradiction.

Suppose $|R\cap V_1|=2$, i.e. $a=m+2$.  Write $R\cap V_1=\{u, w_1\}$ and $R\cap V_2=\{u_2, v_1\}$.  Note that  $a+1\le d(y_0)\le 2m+1$,  which together with  $2m<a+1$ implies that $a=2m$ and $d(y_0)=2m+1$. Using this with  $a=m+2$, we have that  $m=2$. By Claim 1,  either $\{u_2, y_0\}\mapsto w_1$ or $w_1\mapsto \{u_2, y_0\}$.  Without loss of generality, assume that $\{u_2, y_0\}\mapsto w_1$.   In addition,  $y_0\rightarrow \{x_0, x_1\}\rightarrow y_0$. Since $D$ is strong, there exists a vertex in $V_2\setminus \{y_0, u_2\}$ dominated by $w_1$. If $w_1\rightarrow y_1$, then  $w_1y_1x_0u_2x_1y_0w_1$ is a longer cycle than $C$, a contradiction. If $w_1\rightarrow v_1$, then since $D$ is strong, there exists a vertex in $\{x_0, x_1, u\}$ dominated by $v_1$.  It is not difficult to check that $D$ would contain a longer cycle than $C$, a contradiction.

\vskip0.2cm
 {\noindent\bf Subcase 2.2.2.  } $R\cap V_1\mapsto u_2$.
 \vskip0.2cm

By $d(u_2)\ge 2a-2$, we have $|R\cap V_1|\le 2$.
%Suppose that there exists a vertex $x\in V_1\cap R$ such that $x\mapsto u_2$.  Then $\{x, x_0\}\rightarrow u_2$ implies that $\{x, x_0\}$ is a dominating pair. By assumption, $d(x)\ge 2a-2$, $d(x_0)\ge a+1$ or $d(x)\ge a+1$, $d(x_0)\ge 2a-2$. By Claim **, $y_0\nrightarrow x$. Together together with $x\mapsto u_2$ implies that $d(x)\le 2a-2$. It is easy to see that if $y_i\rightarrow x$, then $u_2\nrightarrow x_{i+1}$.
Suppose $|R\cap V_1|=2$, say $R\cap V_1=\{x, w\}$. In this case, $d(u_2)=2a-2$ and $u_2$ and every vertex of $V(C)\cap V_1$ form a 2-cycle and $\{w, x\}\mapsto u_2$, which implies that $\{x, w\}$ is a dominating pair. Since $C$ ia a longest cycle, we have $(V(C)\cap V_2,  \{x, w\})=\emptyset$.   Hence  $d(w)\le m+3=a-2+3=a+1$ and $d(x)\le a+1$, a contradiction to the fact that $\{x, w\}$ is a dominating pair.

Suppose $|R\cap V_1|=1$, say $R\cap V_1=\{x\}$. Observe that for any $y_i\in V(C)\cap V_2$, if $y_i\rightarrow x$, then $u_2\nrightarrow x_{i+1}$. Clearly,   $y_0\nrightarrow x$. Since $D$ is strong, there exists $y_i\in V(C)\cap V_2$ such that $y_i\rightarrow x$. Assume, without loss of generality, that $y_1\rightarrow x$.  The above observation implies  $u_2\nrightarrow x_2$. This together with $x\mapsto u_2$ implies that $d(u_2)=2a-2$ and $u_2$ and every vertex of $V_2\setminus \{x, x_2\}$ form a 2-cycle and $x_2\mapsto u_2$. Continuing using the observation, we have that $y_i\nrightarrow x$, for every $i\neq 1$, that is to say, $d^-(x)=1$. Thus, $a+1\le d(x)\le 1+(m+1)=a+1$. Note that, for any $x_i\in V(C)\cap V_1$, $\{x_i, x\}$ is a dominating pair and so $d(x_i)\ge 2a-2$, in particular, $d(x_2)\ge 2a-2$.

 Suppose $a\ge 4$. Denote $Q=y_1xu_2C[x_3, y_1]$. It is clear that $Q$ is a cycle of length $2m$. By  $\{y_2, u_2\}\rightarrow x_3$,  we have  that  $\{y_2, u_2\}$ is a  dominating pair, which implies that $d(y_2)\ge a+1$. If $y_2\rightarrow x_2$, then, by Lemma \ref{notmerge},  $3a-1\le d(x_2)+d(y_2)=4+d_{V(Q)}(x_2)+d_{V(Q)}(y_2)\le 2m+4=2a+2$ implies that  $a\le 3$, a contradiction. Hence $y_2\nrightarrow x_2$. This together with $d(x_2)\ge 2a-2$ and $x_2\mapsto u_2$ imply that $d(x_2)=2a-2$ and $x_2$ and every vertex of $V_2\setminus \{u_2, y_2\}$ form a 2-cycle.  For any $x_i\in (V_1\cap V(C))\setminus \{x_2, x_3\}$, $y_2\nrightarrow x_i$, for else  $y_2C[x_i, y_1]xu_2C[x_3, y_{i-1}]x_2y_2$ is a hamiltonian cycle, a contradiction. Thus, $d(y_2)\le 2+m=a+1$. Earlier, we showed that $d(y_2)\ge a+1$. So $d(y_2)=a+1$. Moreover $d^+(y_2)=1$ and $d^-(y_2)=a$. However, $y_1xu_2x_1C[y_2, y_0]x_2y_1$ is a hamiltonian cycle, a contradiction.

Now assume that $a=3$. In this case, $2a-2=a+1=4$. By the above argument, we have obtained that $\{x_0, x\}\mapsto u_2$, $x\mapsto y_0$,  $x\rightarrow y_1\rightarrow x$,  $u_2\rightarrow x_1\rightarrow u_2$, $d(y_0)\ge 4$ and $d(x_0)\ge 4$.   If $y_0\rightarrow x_0$, then $y_0x_0u_2x_1y_1xy_0$ is a hamiltonian cycle, a contradiction. Thus, $x_0\mapsto y_0$. This together with $d(x_0)\ge 4$ and $d(y_0)\ge 4$ imply that  $x_0\rightarrow y_1$ and $x_1\rightarrow y_0$. Note that $D$ is isomorphic to the digraph $H_2$.

\vskip0.2cm
 {\noindent\bf Subcase 2.2.3.  } $u_2\mapsto R\cap V_1$.
 \vskip0.2cm

 Since the proof is similar to Subcase 2.2.2,  we omit them. This completes the proof of the theorem.

\end{proof}

{\bf\noindent The proof of Theorem \ref{mainresult}}.

\begin{proof}
  Using Theorem \ref{sharp}, it suffices to consider that $D$ is isomorphic to one of the digraphs $H_1$, $H_2$ and $H_3$. Now we consider the three cases. Suppose that $D$ is isomorphic to the digraph $H_1$. Then $\{x_1, x_2\}\rightarrow y_2$ implies that $\{x_1, x_2\}$ is a dominating pair. Since $a=2$, we have $2a-1=a+1=3$. By assumption, $d(x_1)\ge 3$. So $x_1\rightarrow y_2$ or $y_2\rightarrow x_1$, say $x_1\rightarrow y_2$. Clearly, $D$ is  hamiltonian. Suppose that $D$ is isomorphic to the digraph $H_2$. Then $\{x_1, x_2\}\rightarrow y_1$ implies that $\{x_1, x_2\}$ is a dominating pair. By assumption, without loss of generality, assume that $d(x_1)\ge 2a-1=5$ and $d(x_2)\ge a+1=4$. So $x_1$ must dominate one of $\{y_2, y_3\}$, say $y_2$. Note that $x_1y_2x_2y_1x_3y_3x_1$ is a hamiltonian cycle. Suppose that $D$ is isomorphic to the digraph $H_3$. Then $\{x_1, x_3\}\rightarrow y_2$ implies that $\{x_1, x_3\}$ is a dominating pair. By assumption, without loss of generality, assume that $d(x_1)\ge 2a-1=5$ and $d(x_3)\ge a+1=4$. So $y_2x_1\in A(D)$ or $x_1y_1\in A(D)$. If $y_2x_1\in A(D)$, then $y_2x_1y_3x_2y_1x_3y_2$ is a hamiltonian cycle; if $x_1y_1\in A(D)$, then $x_1y_1x_3y_2x_2y_3x_1$ is a hamiltonian cycle.
\end{proof}

It is natural to propose the following problem.

\begin{problem}
  Consider a strong balanced bipartite digraph on $2a$ vertices where $a\ge 4$. Suppose that $D$ satisfies the condition $B_k$ with $2\le k\le a/2$. Is $D$ hamiltonian?
\end{problem}

%Unfortunately, Theorem \ref{mainresult} has no a general example of sharpness. However, consider a balanced bipartite digraph $D$ on $2a$ vertices with $a\ge 3$ satisfies the condition $B_{\frac{a}{2}}$, there is a general example of sharpness. For example. For $a\ge 3$ and $1\le l<a/2$, let $D(a,l)$ be a bipartite digraph with bipartite set $V_1$ and $V_2$ such that $V_1$ (resp. $V_2$) is a disjoint union of sets $R, S$ (resp. $U, W$) with $|R|=|U|=l$, $|S|=|W|=a-l$, and $A(D(a,l))$ consists of the following arcs:
%\begin{description}
%\item (a) $ry$ and $yr$, for all $r\in R$ and $y\in V_2$,
%\item (b) $ux$ and $xu$, for all $u\in U$ and $x\in V_1$, and
%\item (c) $sw$, for all $s\in S$ and $w\in W$.
%\end{description}

\section{Remark}

If $D$ is a strong balanced bipartite digraph on $2a$ vertices such that $d(x)+d(y)\ge 3a$ for every pair of dominating and every pair of dominated vertices $\{x,y\}$, then the following theorem shows that $D$ contains a cycle factor. However $D$ may contain no hamiltonian cycle, see the digraph $H_1$. Here $a=2$. Note that $\{x_1,x_2\}$ and $\{y_1, y_2\}$ are both dominating  and dominated pairs, where $d(x_1)+d(x_2)=3a$ and $d(y_1)+d(y_2)=3a$. Clearly, $D$ has no hamiltonian cycle.

\begin{proposition} Let $D$ be a strong balanced bipartite digraph on $2a$ vertices. Suppose that $d(x)+d(y)\ge 3a$ for every pair of dominating and every pair of dominated vertices $\{x,y\}$. Then $D$ contains a cycle factor. \end{proposition}

\begin{proof} Let $V_1$ and $V_2$ denote the partite sets of $D$. Observe that $D$ contains a cycle factor if and only if there exist both a perfect matching from $V_1$ to $V_2$ and a perfect matching from $V_2$ to $V_1$. Therefore, by K\"{o}nig-Hall theorem, it suffices to show that $|N^+(S)|\ge |S|$ for every $S\subset V_1$ and $|N^+(T)|\ge |T|$ for every $T\subset V_2$.

For a proof by contradiction, suppose that a non-empty set $S\subset V_1$ is such that $|N^+(S)|<|S|$. Then $V_2\setminus N^+(S)\neq\emptyset$ and for every $y\in V_2\setminus N^+(S)$, we have $d^-(y)\le a-|S|$, hence $$d(y)\le 2a-|S|.\eqno(4.1)$$ If $|S|=1$, say $S=\{x\}$, then $|N^+(S)|=0$, which means $d^+(x)=0$, a contradiction to the fact that $D$ is strong. If $|S|=a$, then $d^-(y)=0$ for every $y\in V_2\setminus N^+(S)$, a contradiction to the fact that $D$ is strong. Thus $2\le |S|\le a-1$.  We now consider the following two cases.
\vskip 0.2cm
{\bf Case 1.} $\frac{a}{2}<|S|\le a-1$.
\vskip 0.2cm
In this case, $|V_2\setminus N^+(S)|\ge 2$ and $|V_1\setminus S|<|V_2\setminus N^+(S)|$. Since $D$ is strong, for every $y\in V_2\setminus N^+(S)$, $d^-(y)\ge 1$. This together with $|V_1\setminus S|<|V_2\setminus N^+(S)|$ implies that there exist $y_1, y_2\in V_2\setminus N^+(S)$ such that $\{y_1, y_2\}$ is a dominated pair. By assumption, $d(y_1)+d(y_2)\ge 3a$. But, by (4.1), $d(y_1)+d(y_2)\le 2(2a-|S|)=4a-2|S|<3a$, a contradiction.

\vskip 0.2cm
{\bf Case 2.} $2\le |S|\le \frac{a}{2}$.
\vskip 0.2cm

If this is so, then, for every $x\in S$, we have $$d(x)=d^-(x)+d^+(x)\le a+(|S|-1)\le \frac{3a}{2}-1.\eqno(4.2)$$ Since $D$ is strong, for every $x\in S$, $d^+(x)\ge 1$. This together with $|N^+(S)|<|S|$ implies that there exist $x_1, x_2\in S$ such that $\{x_1,x_2\}$ is a dominating pair. By assumption, $d(x_1)+d(x_2)\ge 3a$. But, by (4.2),$d(x_1)+d(x_2)\le 3a-2$, a contradiction.

This completes the proof of existence of a perfect matching from $V_1$ to $V_2$. The proof for a matching in the opposite direction is analogous. The proof of the theorem is complete.\end{proof}

\section{Acknowledgements} 

This work is supported by the National Natural Science Foundation for Young Scientists of China\\ (11401354)(11501341)(11401353)(11501490).

\end{document}